\documentclass[smallextended,envcountsect]{svjour3} 

\usepackage{color}
\usepackage{array,graphicx}
\usepackage{psfrag}
\usepackage{subfigure}
\usepackage{url}
\usepackage{cite}
\usepackage{epsfig}
\usepackage{amsmath,amssymb}
\newcommand{\eqdef}{\overset{\text{def}}{=}}
\newcommand{\Prob}{\mathbf{P}}
\newcommand{\Exp}[1]{\mathbf{E}\left[#1\right]}
\newcommand{\E}{\ensuremath{\mathbf{E}}}

\newcommand{\rset}{\mathbb{R}}
\newcommand{\nset}{\mathbb{N}}

\providecommand{\norm}[1]{\lVert#1\rVert}
\newcommand{\prox}{\ensuremath{\operatorname{prox}}}

\usepackage[colorinlistoftodos,bordercolor=orange,
backgroundcolor=orange!20,linecolor=orange,textsize=scriptsize]{todonotes}

\newtheorem{assumption}[theorem]{Assumption}

\newtheorem{algorithm}[theorem]{Algorithm}

\begin{document}

\title{  General convergence analysis of stochastic first order methods for composite  optimization }


\author{Ion Necoara}

\institute{Ion Necoara \at
Automatic Control and Systems Engineering Department, University  Politehnica Bucharest,\\        060042 Bucharest, Romania, ion.necoara@acse.pub.ro           
}

\date{}

\maketitle

\begin{abstract}
 In this paper we consider stochastic composite convex optimization problems with the objective function satisfying a stochastic bounded gradient condition, with or without a quadratic functional growth property.  These models include the most well-known  classes of objective functions analyzed in the literature:  non-smooth Lipschitz  functions and composition of a (potentially) non-smooth function  and a smooth function, with or without strong convexity. Based on the flexibility offered by our optimization model we consider several variants of stochastic first order methods, such as   the stochastic proximal gradient and the stochastic proximal point algorithms. Usually, the  convergence theory  for these methods has been  derived for simple stochastic optimization models satisfying restrictive assumptions, the rates are in general  sublinear and hold only for specific  decreasing  stepsizes.  Hence, we analyze the convergence rates of stochastic first order  methods with constant or variable stepsize under general assumptions covering a large class of objective functions. For constant stepsize we show that these methods can achieve linear convergence rate  up to a constant proportional to the stepsize and under some strong stochastic bounded gradient condition even pure linear convergence.  Moreover,    when a variable stepsize is chosen we derive sublinear convergence rates for these stochastic first order methods. Finally, the stochastic gradient mapping and the Moreau smoothing mapping introduced in the present paper lead to simple and intuitive proofs.
\end{abstract}



\section{Introduction}
The randomness in  most of the practical optimization
applications  led the stochastic optimization field to become an
essential tool for many applied mathematics areas, such as machine
learning and statistics \cite{MouBac:11}, distributed control and
signal processing \cite{NecNed:11}, sensor networks
\cite{BlaHer:06} and others. In particular, in statistics and machine learning
applications the optimization algorithms involve numerical
computation of parameters for a system designed to make decisions
based on yet unseen data. The recent success of certain optimization
methods for statistics and machine learning problems has motivated increasingly
great efforts into developments of new numerical algorithms or  into
analyzing deeper the existing ones.

In this paper we analyze   a popular class of  algorithms for solving stochastic composite convex optimization problems, that is  stochastic first order (SFO) methods \cite{AtcFor:14,MouBac:11,LacSch:12,Lan:12,NedBer:00,NemJud:09,PatNec:17,SchRou:13}. We assume that we have access to an  unbiased estimate of the gradient or of the proximal operator of a certain function we wish to minimize, which is key to scale up optimization and to address streaming settings where data arrive in time. In these scenarios,  SFO methods
independently sample an  unbiased estimate of the gradient or of the proximal
operator and then take a step along this direction with a certain
stepsize length. These algorithms  are typically the methods of choice in
practice for many applications due to their cheap
iteration and superior empirical performance. However, the
theoretical convergence rates from the literature are  usually sublinear and hold only for decreasing stepsizes. Moreover, the convergence theory treats separately  smooth or non-smooth objective functions and covers  usually unconstrained   optimization models.  On the other hand, in this paper we present  a general framework for the analysis of  SFO  algorithms for solving  general composite optimization problems, expressed in terms of expectation operator. This framework is based on the assumptions that   the objective function satisfies a stochastic bounded gradient condition,  with or without a quadratic functional growth property.  These conditions include the most well-known  classes of objective functions analyzed in the literature:  non-smooth Lipschitz  functions and composition of a (potentially) non-smooth function  and a smooth function, with or without strong convexity. Based on this framework we derive a complete convergence analysis for these SFO methods, that is stochastic proximal gradient and proximal point algorithms.

More specifically, a very popular approach for solving
stochastic optimization problems, where the regularization function
is the indicator function of some simple convex set,  is the
stochastic gradient descent (SGD) algorithm
\cite{LacSch:12,NedBer:00,NemJud:09,SchRou:13}.  In this paper we complement and extend the previous results for SGD to a wider class of stochastic composite optimization problems  having more general assumptions  and to  more general stochastic first order methods.   In particular, we extend the convergence analysis of this well-known
method to composite optimization problems  having a smooth or non-smooth term and a general regularization term that leads to the stochastic proximal gradient (SPG) algorithm.  Other results related to the  convergence behavior of SPG can be found in
\cite{AtcFor:14,Lan:12,RosVil:14}. Furthermore, despite the
fact that the computational performance of SPG  may be good under
certain circumstances, there is recent evidence of its instability
for unappropriate parameters choice \cite{MouBac:11}. To avoid this
behavior of SPG scheme,  we also analyze the convergence behavior of
stochastic proximal point (SPP) algorithm. Papers on SPP related to
our work are e.g. \cite{RyuBoy:16,TouTra:16,PatNec:17}. We show that  these two methods, SPG and SPP,  can achieve linear convergence rate up to a constant proportional to the stepsize and under some strong stochastic bounded gradient condition even pure linear convergence. We also prove that the strong stochastic bounded gradient condition is not only sufficient but also necessary for obtaining linear convergence.  When the strong stochastic bounded gradient condition does not hold we show that restarted variants of these methods can still achieve linear convergence. Moreover, when variable stepsize is chosen we derive sublinear convergence rates for these stochastic first order methods. Finally, the stochastic gradient mapping and the Moreau smoothing mapping from this present paper lead to more elegant and intuitive proofs.

\noindent \textit{Content}: This section continues with the presentation of our problem of interest and the main assumptions. In Section 2  we propose two stochastic first order methods and derive their  convergence rates.  Finally, Section 3 presents a  restarted variant of these methods and  analyze its convergence.

\vspace{-0.3cm}

\subsection{Problem formulation}
\vspace{-0.3cm}
\noindent Let   $f, g \colon \rset^m \times \Omega \to \bar{\rset}$ be  two proper convex
 functions in the first argument, where $\Omega$ is endowed with a
probability distribution $\Prob$.  Then, we consider the following
general composite stochastic optimization problem:
\begin{align}
\label{eq_problem}
F^* & = \min_{x \in \rset^m} F(x) \quad \left( \eqdef f(x) + g(x) \right),
\end{align}
where the two functions have stochastic representations in the form of expectation, i.e.  $f(x) = \Exp{f(x, \xi)}$ and $ g(x) = \Exp{g(x, \xi)}$. We   assume that the expectation taken with respect to the random variable $\xi \in \Omega$ is finite
for all $x \in \textnormal{dom} \ F$. Hence we consider a flexible splitting of our stochastic objective function $F(x) =   \Exp{F(x, \xi)}$ in  composite form  as a sum of two  terms: 
$$ F(x,\xi) = f(x,\xi) + g(x,\xi). $$   In this paper we assume $f(\cdot,\xi)$ either differentiable or nondifferentiable function and we use,  with some abuse of notation, the same notation for the gradient or subgradient of $f(\cdot,\xi)$ at $x$, that is $\nabla f(x,\xi) \in \partial f(x, \xi)$, where $\partial f(x,\xi)$ is either a singleton or a nonempty convex set. The other term, $g(x,\xi)$,  is assumed to be simple and its structure known.  We usually refer to $g$ as the regularizer and use the same notation as before  $\nabla g(x,\xi) \in \partial g(x, \xi)$ for a (sub)gradient of $g(x, \xi)$ at $x$.  We  also assume in the sequel that to have access to either  an unbiased
stochastic estimates of the (sub)gradients of function $f$, i.e.:
\[ \mathbb{E}[\nabla f(x,\xi)] \in \partial f(x), \;\; \text{where}  \;\; \nabla f(x,\xi) \in \partial f(x,\xi),  \] or to
stochastic estimates of the proximal operator of $f$, i.e.:
\[   \prox_{\gamma f(\cdot,\xi)}(x) = \arg
\min\limits_{y \in \rset^m} f(y,\xi) + \frac{1}{2 \gamma} \|y
-x\|^2.\]
We refer to $\prox_{\gamma f (\cdot,\xi)}(x)$ as the proximal operator of the
function $f(\cdot,\xi)$ at $x$ with stepsize $\gamma$.  We  always assume to have access  to  stochastic estimates of the proximal operator of $g$, i.e. $ \prox_{\gamma g(\cdot,\xi)}(x) $. These are  key assumptions in our stochastic optimization to scale up the numerical algorithms. Template \eqref{eq_problem}
covers many applications in machine learning, statistics,  signal
processing, control and other areas,  by appropriately choosing the 
functions $f$ and $g$.

\vspace{-0.3cm}

\subsection{Main assumptions}
\vspace{-0.3cm}
\noindent Let us denote by $X^*$ the optimal set of the convex problem \eqref{eq_problem} and for any $x \in \rset^m$ we denote its projection onto $X^*$ by $\overline{x}$, that is $\overline{x} = \Pi_{X^*}(x)$. In this paper we consider additionally the following assumptions:
\begin{assumption}
\label{assumption1}
The stochastic (sub)gradients
of $F$ satisfy a stochastic bounded gradient condition restricted on
any segment $[x, \; \overline{x}]$, that is there exist nonnegative
constants  $L \geq 0$ and $B \geq 0$ such that:
\begin{equation}
\label{as:main1_spg} B^2 +  L(F(x) - F(\overline{x})) \geq
\E_{\xi}[\| \nabla F(x,\xi)  \|^2 ] \quad  \forall x \in \textnormal{dom} \ F.
\end{equation}
\end{assumption}

\begin{assumption}
\label{assumption2}
The function  $F$ satisfies a quadratic functional growth condition restricted on
any segment $[x, \; \overline{x}]$, that is there exists  $\mu \geq 0$ such that:
\begin{equation}
\label{as:strong_spg} F(x) - F(\overline{x}) \geq  \frac{\mu}{2} \|x
- \overline{x}\|^2 \quad \forall x \in \textnormal{dom} \ F.
\end{equation}
\end{assumption}

\noindent Note that none of the conditions  \eqref{as:main1_spg} and/or \eqref{as:strong_spg} implies convexity of the function $F$.  When $B=0$ we refer to \eqref{as:main1_spg}  as the \textit{strong stochastic bounded gradient} condition. Note that the strong stochastic  bounded gradient condition has been considered in \cite{NecRic:16} for proving linear convergence of stochastic gradient descent method for solving convex feasibility problems.  In this paper we generalize the strong stochastic bounded gradient condition  from  \cite{NecRic:16} by adding   a positive constant $B$ in order to  cover  the most well-known classes of objective functions analyzed in the literature: non-smooth Lipschitz functions, and composition of a (potentially) non-smooth function and a smooth function. A similar stochastic bounded gradient condition of the form \eqref{as:main1_spg} has been also considered in \cite{DucSin:09} in the context of non-smooth  optimization. 
Further, the quadratic functional growth condition \eqref{as:strong_spg}   has been  considered in \cite{NecNes:15} for deterministic optimization,  where it was proved  that the class of smooth functions satisfying \eqref{as:strong_spg} is the largest one for which gradient method is converging linearly.  We now present several classes of
functions satisfying Assumptions \ref{assumption1} and
\ref{assumption2}. First note that  Assumption \ref{assumption1}  is very general and covers a large class of functionals, such as Lipschitz  functions or functions having Lipschitz continuous gradients.

\noindent \textbf{Example 1} [Non-smooth (Lipschitz) functions satisfy Assumption \ref{assumption1}]:
Assume that the functions $f(\cdot,\xi)$ and $g(\cdot,\xi)$ have bounded
(sub)gradients:
\[ \|\nabla f(x,\xi)\| \leq B_f  \quad \text{and} \quad
\|\nabla g(x,\xi) \| \leq B_g \quad \forall x \in \text{dom} \ F.  \]
Then, obviously Assumption \ref{assumption1} holds with $ L=0 \quad  \text{and} \quad B^2 = 2 B_f^2
+ 2 B_g^2.$

\noindent \textbf{Example 2} [Smooth (Lipschitz gradient) functions satisfy Assumption \ref{assumption1}]:  Condition  \eqref{as:main1_spg}  contains the class of  functions formed as a sum of two terms, one   having Lipschitz continuous gradient  and the other having bounded subgradients. Indeed, let us assume that  $f(\cdot,\xi)$ has  Lipschitz continuous 
gradient, i.e. there exists $L(\xi)>0$ such that:
\[  \|\nabla f(x, \xi) - \nabla f(\bar{x},\xi)\| \leq L(\xi) \|x - \bar{x}\| \quad \forall x \in \text{dom} \ F. \]
Then, using standard arguments we have \cite{Nes:04}:
\[  f(x, \xi) -  f(\bar{x},\xi) \geq  \langle \nabla f(\bar{x}, \xi), x -\bar{x} \rangle    + \frac{1}{2L(\xi)} \|\nabla f(x, \xi) - \nabla f(\bar{x}, \xi) \|^2. \]
Assuming that $L(\xi) \leq L_f$ for all $\xi  \in \Omega$ and $g(\cdot,\xi)$  convex, then adding  $g(x,\xi) - g(\bar{x},\xi) \geq \langle \nabla g(\bar{x},\xi), x -\bar{x} \rangle $ in the previous inequality, where  $\nabla g(\bar{x},\xi) \in \partial g(\bar{x},\xi)$, and then taking expectation  w.r.t. $\xi$, we get:
\begin{align*}
F(x) -  F(\bar{x})  & \geq  \langle \nabla F(\bar{x}), x -\bar{x} \rangle    + \frac{1}{2L_f} \Exp{\|\nabla f(x, \xi) - \nabla f(\bar{x}, \xi) \|^2},
\end{align*}
where we used that $\nabla f(\bar{x},\xi)$ and $\nabla g(\bar{x},\xi)$ are unbiased stochastic estimates of the (sub)gradients of $f$ and $g$ and thus $\nabla F(\bar{x}) = \Exp{\nabla g(\bar{x},\xi) + \nabla g(\bar{x},\xi)} \in \partial F(\bar{x})$.    Using the optimality conditions for  \eqref{eq_problem} in $\bar{x}$, i.e.   $0 \in \partial F(\bar{x})$,  we get:
\begin{align*}
F(x) - F^*  & \geq \frac{1}{2L_f} \Exp{\|\nabla f(x, \xi) - \nabla f(\bar{x}, \xi) \|^2}.
\end{align*}
Therefore, for any $\nabla g(x,\xi) \in \partial g(x,\xi)$ we have:
\begin{align*}
 \Exp{\|\nabla F(x, \xi) \|^2} & =   \Exp{\|\nabla f(x, \xi) - \nabla f(\bar{x}, \xi) +\nabla g(x,\xi) + \nabla f(\bar{x}, \xi) \|^2} \notag\\
&\leq 2 \E[\|\nabla f(x, \xi) - \nabla f(\bar{x}, \xi)\|^2 ] +2 \E[\|\nabla g(x,\xi) + \nabla f(\bar{x}, \xi)\|^2 ]\notag\\
&\leq 4L_f (F(x) - F^*) +2 \E[\|\nabla g(x,\xi)  + \nabla f(\bar{x}, \xi)\|^2 ].
\end{align*}
Assuming now that  the regularization function $g(x,\xi)$ has bounded
subgradients, i.e. $\|\nabla g(x,\xi)\| \leq B_g$, then we get that the stochastic bounded gradient condition  \eqref{as:main1_spg} holds with:
\[  L =  4 L_f \quad \text{and} \quad B^2 =  4  (B_g^2 +\min_{\bar{x} \in X^*} \E[\|\nabla f(\bar{x}, \xi)\|^2]). \]

\noindent Further,  many practical problems satisfy the quadratic functional growth condition \eqref{as:strong_spg}, the most relevant one is given next.

\noindent \textbf{Example 3} [Composition between a strongly convex function and a linear map satisfy Assumption \ref{assumption2}]: Assume $F(x)= \hat f(A^Tx) + g(x)$, where $\hat f$ is a strongly convex function with constant $\sigma_f>0$, $A$ is a matrix of appropriate dimension and $g$ is a polyhedral function. Since  $g$ has a polyhedral
epigraph, then the  optimization problem \eqref{eq_problem} can be equivalently written as:
\[  \min_{x,\zeta} \hat f(A^Tx) + \zeta \qquad \text{s.t.}: \quad Cx + c \zeta \leq d, \]
for some appropriate  matrix $C$ and vectors $c$ and $d$ of
appropriate dimensions. In conclusion, this reformulation leads to the following extended  problem:
\[  \hat{F^*} = \min_{\hat{x}=[x^T \; \zeta]^T}  \hat F(\hat{x}) \quad \left(= \hat f(\hat{A} \hat{x}) + \hat{c}^T \hat{x} \right) \qquad \text{s.t.}: \quad \hat{C} \hat{x}  \leq \hat{d}, \]
where $\hat{A}=[A \; 0], \; \hat{c} =[0 \; 1]^T, \;\hat{C} = [C \; c]$ and $\hat{d} =d$. It can be easily seen that $\hat{x}^* = [(x^*)^T \; \zeta^*]^T$ is an optimal point of this extended optimization problem if $x^*$ is optimal for the original problem and $g(x^*) = \zeta^*$. Moreover, we have $\hat{F^*}= F^*$.  Following a standard argument, as e.g. in \cite{NecNes:15}, there
exist $b^*$ and $s^*$ such that the optimal set of the extended optimization problem  is given by $\hat{X}^* =\{\hat{x}: \; \hat{A}  \hat{x} = b^*, \; \hat{c}^T \hat{x}= s^*, \; \hat{C} \hat{x}  \leq \hat{d} \}$. Further, since $\hat f$ is  strongly convex function with constant $\sigma_f>0$, it follows from \cite{NecNes:15}[Theorem 10] that for any $M>0$ on any sublevel set defined in terms of  $M$ the function  $\hat{F}$ satisfies a quadratic functional growth condition  of the form:
\[  \hat{F}(\hat{x}) - \hat{F^*} \geq \frac{\mu(M)}{2} \|  \hat{x} - \hat{x}^* \|^2 \quad \forall \hat{x}: \; \hat{F}(\hat{x}) - \hat{F^*} \leq M,  \]
where  $\mu(M) = \frac{\sigma_f}{\theta^2 (1 + M \sigma_f + 2 \|\nabla \hat f(\hat{A} \hat{x}^*)\|^2)}$,   with $\hat{x}^* \in \hat{X}^*$ and $\theta$ is the Hoffman bound for the optimal polyhedral set $\hat{X}^*$. Now, setting $\zeta = g(x)$ in the previous inequality, we get  $F(x) - F(\overline{x}) \!\geq\! \frac{\mu(M)}{2} ( \| x - \overline{x}\|^2 + (\zeta - \zeta^*)^2) \!\geq\!  \frac{\mu(M)}{2} \| x - \overline{x}\|^2$ for all $x:  F(x) - F(\overline{x}) \!\leq\!  M$.
In conclusion,  the objective function  $F$ satisfies the quadratic functional growth condition \eqref{as:strong_spg} on any sublevel set, that is for any $M>0$ there exists $\mu(M)>0$ defined above such that:
\[  F(x) - F(\overline{x}) \geq  \frac{\mu(M)}{2} \| x - \overline{x}\|^2
\quad \forall x: \;  F(x) - F(\overline{x}) \leq  M. \]
The quadratic functional growth condition \eqref{as:strong_spg} is a relaxation of strong convexity notion, see \cite{NecNes:15} for a more detailed
discussion. Clearly, any strongly convex function $F$ satisfies \eqref{as:strong_spg} (see e.g. \cite{Nes:04}).

\vspace{-0.4cm}

\subsection{Preliminaries}
\vspace{-0.3cm}
The following results  are  also useful in the sequel for proving sublinear convergence rates for the stochastic first order methods that we analyze in this paper. They are simple adaptations of some  standard recurrences for $t \geq 0$ to $t \geq t_0$, for some finite index  $t_0$,   see e.g. \cite{Pol:87,NedBer:00,MouBac:11,PatNec:17,RosVil:14}.

\begin{lemma} [Lemma 4, \cite{Pol:87}]
\label{lem_polyak1} If there exist constants $c, d>0$ and  finite
index  $t_0 \geq 0$ such that the nonnegative sequence $r_t$ has
$r_{t_0}$ finite and satisfies the recurrence:
\[ r_{t+1} \leq \left( 1  - \frac{c}{t+1}\right) r_t + \frac{d}{(t+1)^2} \qquad \forall t
\geq t_0, \]  then $r_t$ can be bounded as:
\begin{align*}
r_{t} \leq
\left\{\begin{array}{lr}
\frac{ 2t_0r_{t_0} + 2d(1+ \log((t+1)/(t_0+1)))}{t+1}& \text{if}  \quad c=1 \\
\left( r_{t_0} (t_0+1)^c + \frac{2d(2-c)(t_0+1)^c}{1-c}\right)
 \frac{1}{(t+1)^c} & \text{if} \quad c<1 \\
\left((t_0+1)r_{t_0} + \frac{d}{c-1}\right) \frac{1}{t+1} &
\text{if}  \quad  c >1
\end{array}\right.
\quad \forall t \geq t_0.
\end{align*}
\end{lemma}

\begin{lemma} [Lemma 5, \cite{Pol:87}; Theorem 14, \cite{PatNec:17}]
\label{lem_polyak2} If there exist constants $c, d>0$, $\gamma \in (0, 1)$ and $\zeta > \gamma$ and  finite index  $t_0 \geq 0$ such that the nonnegative sequence $r_t$ has
$r_{t_0}$ finite and satisfies the recurrence:
\[ r_{t+1} \leq \left( 1  - \frac{c}{(t+1)^\gamma}\right) r_t + \frac{d}{(t+1)^\zeta} \qquad \forall t \geq t_0, \]  then $r_t$ can be bounded as:
$
r_{t} \leq  {\cal O}\left(  \frac{1}{t^{\zeta-\gamma}} \right) r_{t_0} \quad  \forall t \geq t_0.
$
\end{lemma}

\vspace{-0.4cm}

\section{Stochastic first order methods}
Our general composite stochastic optimization problem \eqref{eq_problem} is  flexible, allowing us to deal with objective functions having specific structures. Specifically, we can assume that the first term  $f(\cdot,\xi)$  is  given by a black-box (sub)gradient oracle or is proximal friendly.  Based on these properties we propose two stochastic first order methods for solving the composite optimization problem, that is   the stochastic proximal gradient and the stochastic proximal point algorithms.

\vspace{-0.2cm}

\subsection{Stochastic proximal gradient (SPG)}
\vspace{-0.3cm}
In this section  we consider stochastic composite objective functions formed as a sum of two convex terms: first term, $f(\cdot,\xi)$,  is  given by a black-box (sub)gradient oracle and consequently we have access to  unbiased stochastic estimates of the (sub)gradients of $f$ in the sense that $\mathbb{E}[\nabla f (x,\xi)]  \in  \partial f(x)$, and second term, $g(\cdot,\xi)$,  admits an easily computable proximal mapping. Regularizers $g(\cdot,\xi)$ that admit closed form solution of the prox operator is e.g.  Lasso type $\lambda \|x\|_1$ or elastic net $\lambda_1 \|x\|^2 + \lambda_2 \|x\|_1$ for which the prox can be computed via a soft-thresholding function.  Therefore,  for solving the general composite problem  \eqref{eq_problem}  we consider the stochastic proximal gradient (SPG) algorithm:
\begin{algorithm}
Let $(\xi_t)_{t \in \nset}$ be an i.i.d sequence and $x_0 \in
\rset^m$. Iterate:
\begin{equation}
\label{eq_spgm}  \nabla f(x_t,\xi_t) \in\partial f(x_t,\xi_t) \quad  \text{and} \quad x_{t+1} = \prox_{\gamma_t g(\cdot,\xi_t)}(x_t - \gamma_t \nabla f(x_t,\xi_t) ),
\end{equation}
where $(\gamma_t)_{t \in \nset}$ is a strictly positive sequence of
stepsizes.
\end{algorithm}

\noindent Note that when  $g(\cdot,\xi)$ is the indicator function of a simple
non-empty closed convex set $C_\xi$, that is $g(x,\xi)=\textbf{1}_{C_\xi}(x)$, then the
previous SPG algorithm becomes a stochastic projected  (sub)gradient
descent method \cite{LacSch:12,NedBer:00,NemJud:09}: $$x_{t+1} = \Pi_{C_{\xi_t}}(x_t
- \gamma_t \nabla f(x_t,\xi_t)).$$  Let us define the stochastic
gradient mapping (for simplicity we omit its dependence on stepsize~$\gamma$):
\[  \mathcal{G}(x;\xi) = \gamma^{-1} \left( x - \prox_{\gamma g(\cdot,\xi)}(x - \gamma {\nabla} f(x,\xi)) \right). \]
Then, it follows immediately that  the previous stochastic proximal
gradient iteration can  be written as:
\[ x_{t+1} = x_t - \gamma_t \mathcal{G}(x_t;\xi_t).  \]
Moreover, from  optimality condition of prox operator there exists
$\nabla g(x_{t+1},\xi_t) \in \partial g(x_{t+1},\xi_t)$ such that:
\[  \mathcal{G}(x_t;\xi_t) =  \nabla f(x_t,\xi_t)  + \nabla g(x_{t+1},\xi_t).  \]

\noindent Denote  also $\xi_{[t]} = \{\xi_0,\ldots, \xi_{t}\} $.  The next theorem provides a descent property for  SPG and for the  proof we use as main tool  stochastic
gradient mapping~$\mathcal{G}(\cdot)$.

\begin{theorem}
\label{th:spg_basic} Let  $f(\cdot, \xi)$ and $g(\cdot, \xi)$ be  convex functions such that $g$ admits an easily computable proximal operator. Additionally,  assume that the stochastic bounded gradient condition from   Assumption \ref{assumption1} holds. Then, for any $t \geq 0$ and  stepsize $\gamma_t >0$, we have the following recursion  for the   SPG iteration:
\begin{align}
\label{spg_basic} & \Exp{\|x_{t+1} - \overline{x}_{t+1} \|^2}  \leq
\Exp{\| x_t - \overline{x}_{t} \|^2} -  \gamma_t  (2 - \gamma_t L)
\Exp{ F(x_t) - F(\overline{x}_{t})} + \gamma_t^2 B^2.
\end{align}
\end{theorem}

\begin{proof}
From the definition of SPG iteration we
have:
\begin{align}
\label{der_spg_basic}
& \| x_{t+1} - \overline{x}_{t+1} \|^2  \leq   \| x_{t+1} - \overline{x}_{t} \|^2 =  \| x_{t} - \overline{x}_{t} - \gamma_t \mathcal{G}(x_t;\xi_t)\|^2 \\
& =  \| x_{t} - \overline{x}_{t} \|^2  - 2 \gamma_t \langle \mathcal{G}(x_t;\xi_t), x_t  - \overline{x}_{t} \rangle + \gamma_t^2 \| \mathcal{G}(x_t;\xi_t) \|^2  \nonumber  \\
& = \| x_{t} - \overline{x}_{t} \|^2  - 2 \gamma_t \langle \nabla
f(x_t,\xi_t) + \nabla g(x_{t+1},\xi_t), x_t  - \overline{x}_{t} \rangle +
\gamma_t^2 \| \mathcal{G}(x_t;\xi_t) \|^2.\nonumber 
\end{align}
Now, we refine the second term. First, from convexity of $f$ we
have:
\[ \langle \nabla f(x_t,\xi_t), x_t  - \overline{x}_{t} \rangle \geq f(x_t,\xi_t) - f(\overline{x}_{t},\xi_t). \]
Then, from  convexity of $g$ and the definition of  stochastic
gradient mapping $\mathcal{G}(\cdot)$, we have:
\begin{align*}
&\langle \nabla g(x_{t+1},\xi_t), x_t  - \overline{x}_{t} \rangle  = \langle  \nabla g(x_{t+1},\xi_t), x_t  - x_{t+1} \rangle + \langle  \nabla g(x_{t+1},\xi_t), x_{t+1}  - \overline{x}_{t} \rangle \\
& \geq \gamma_t \| \mathcal{G}(x_t;\xi_t) \|^2 - \gamma_t \langle \nabla f(x_t,\xi_t), \mathcal{G}(x_t;\xi_t) \rangle + g(x_{t+1}, \xi_t) - g(\overline{x}_t,\xi_t)\\
& \geq \gamma_t \| \mathcal{G}(x_t;\xi_t) \|^2 - \gamma_t \langle
\nabla f(x_t,\xi_t) + \nabla g(x_t,\xi_t), \mathcal{G}(x_t;\xi_t) \rangle +
g(x_{t}, \xi_t) - g(\overline{x}_t, \xi_t).
\end{align*}
Replacing the previous two inequalities in \eqref{der_spg_basic}, we
obtain:
\begin{align*}
\| x_{t+1} - \overline{x}_{t+1} \|^2 & \leq  \| x_{t} -
\overline{x}_{t} \|^2
 - 2 \gamma_t  \left( f(x_t,\xi_t) + g(x_t, \xi_t) - f(\overline{x}_{t},\xi_t) - g(\overline{x}_{t}, \xi_t)\right)  \\
& \quad + 2 \gamma_t^2 \langle \nabla f(x_t,\xi_t) + \nabla g(x_t, \xi_t),
\mathcal{G}(x_t;\xi_t) \rangle - \gamma_t^2 \|
\mathcal{G}(x_t;\xi_t) \|^2.
\end{align*}
Since $2\langle u , v \rangle - \|v\|^2 \leq \|u\|^2$ for
all $v \in \rset^n$ and using that $ F(x_t,\xi_t)  =  f(x_t,\xi_t) +  g(x_t, \xi_t)$ and $\nabla F(x_t,\xi_t)  = \nabla f(x_t,\xi_t) + \nabla g(x_t, \xi_t)$, we further get:
\begin{align*}
\| x_{t+1} - \overline{x}_{t+1} \|^2 & \leq \| x_{t} -
\overline{x}_{t} \|^2
 - 2 \gamma_t \left(  F(x_t,\xi_t)  - F(\overline{x}_{t},\xi_t) \right)  +  \gamma_t^2 \| \nabla F(x_t,\xi_t)\|^2.
\end{align*}
Note that $x_t$ depends only on history $\xi_{[t-1]} =
\{\xi_0,\ldots, \xi_{t-1}\} $, not $\xi_t$. It follows from the
basic property of  conditional expectation that:
\begin{align*}
& \E_{\xi_{[t]}} [\|x_{t+1}  - \overline{x}_{t+1}\|^2] \\
& \leq \E_{\xi_{[t-1]}}[\|x_{t} - \overline{x}_{t}\|^2]  - 2 \gamma_t  \E_{\xi_{[t-1]}}
\left[ \E_{\xi_t} [ F(x_t,\xi_t) - F(\overline{x}_{t},\xi_t)  | \xi_{[t-1]}
]\right] \\
& \qquad +  \gamma_t^2 \E_{\xi_{[t]}}  \left[\| \nabla F(x_t,\xi_t) \|^2 ] \right]\\
& =  \E_{\xi_{[t-1]}}[\|x_{t} \!- \overline{x}_{t}\|^2]  - 2
\gamma_t \E_{\xi_{[t-1]}} \left[ F(x_t) - F(\overline{x}_{t})\right]
+ \gamma_t^2 \E_{\xi_{[t]}}  \left[\| \nabla F(x_t,\xi_t) \|^2 \right].
\end{align*}
Further, making use of the stochastic bounded gradient condition given in  Assumption
\ref{assumption1} in the previous relation we get:
\begin{align*}
& \E_{\xi_{[t]}} [\|x_{t+1}  - \overline{x}_{t+1}\|^2] \\
& \leq \E_{\xi_{[t-1]}}[\|x_{t} \!- \overline{x}_{t}\|^2]  - 2
\gamma_t \E_{\xi_{[t-1]}} \left[ F(x_t) - F(\overline{x}_{t})\right]
+ \gamma_t^2 \E_{\xi_{[t]}}  \left[\| \nabla F(x_t,\xi_t) \|^2
\right] \\
& \overset{\eqref{as:main1_spg}}{\leq} \! \E_{\xi_{[t-1]}}[\|x_{t}
- \overline{x}_{t}\|^2]  - 2 \gamma_t  \E_{\xi_{[t-1]}} \left[
F(x_t) - F(\overline{x}_{t})\right] \\
& \qquad + \gamma_t^2 \left( B^2 + L
\E_{\xi_{[t-1]}} \!\left[F(x_t) - F(\overline{x}_{t})\right] \right)
\\
& = \E_{\xi_{[t-1]}}[\|x_{t} \!- \overline{x}_{t}\|^2]  - \gamma_t
(2 - \gamma_t L) \E_{\xi_{[t-1]}} \left[ F(x_t) -
F(\overline{x}_{t})\right] + \gamma_t^2 B^2,
\end{align*}
which, omitting the dependence  of expectation on $\xi_{[t]}$,  proves the  statement of the theorem. \qed
\end{proof}

\vspace{-0.3cm}

\subsection{Stochastic proximal point (SPP)}
\vspace{-0.3cm}
\noindent  Flexibility in the general optimization problem  \eqref{eq_problem} allow  us to also   consider   simple proximal convex functions $f(\cdot, \xi)$ and $g(\cdot, \xi)$.    In this case it is much better  to use the entire function than just the gradient. Typical examples of functions $f(\cdot, \xi)$ that admit easily computable prox operators are e.g. the quadratic function $(z_\xi^Tx - y_\xi)^2$ or hinge loss $\max(0, 1-y_\xi z_\xi^T x)$ whose prox can be computed in closed form in  ${\cal O}(n)$ operations, or logistic function $\log (1 + e^{-y_\xi z_\xi^T x})$ whose prox operator does not have closed form expression,  but it can be computed very efficiently using Newton iteration on a univariate optimization problem.  Therefore, we also present an algorithm, which we call stochastic proximal point (SPP), where at each iteration we first compute the proximal mapping with respect to the
given  function  $f(\cdot, \xi)$ to the previous iterate and then we perform
the same strategy for the regularization function  $g(\cdot, \xi)$. More precisely,
for solving the general composite problem \eqref{eq_problem}  we consider the following
 SPP algorithm:
\begin{algorithm}
Let $(\xi_t)_{t \in \nset}$ be an i.i.d sequence, and $x_0 \in
\rset^m$. Iterate:
\begin{equation}
\label{eq_sppm}  x_{t+1/2} = \prox_{\gamma_t
f(\cdot,\xi_t)}(x_t) \quad \text{and} \quad x_{t+1}  =
\prox_{\gamma_t g}(x_{t+1/2}),
\end{equation}
where $(\gamma_t)_{t \in \nset}$ is a strictly positive sequence of
stepsizes.
\end{algorithm}

\noindent Note that  SPG can be viewed as   SPP method applied  to the
linearization of $f(z;\xi)$ in $x$, that is to the linear function:
$\ell_f(z;x,\xi) = f(x;\xi) + \langle  \nabla f(x;\xi), z - x \rangle.$
Of course, when $f$ has an easily computable proximal
operator, it is natural to use $f$ instead of its linearization
$\ell_f$.  Let us first derive some basic property for  the $\prox$ operator.
We first recall that  any strongly convex function $h$ with
convexity constant $\sigma_h$ and having the optimal point $x^*$
satisfies the following inequality~\cite{Nes:04}: $h(x) \geq  h(x^*) + \frac{\sigma_h}{2}\|x - x^*\|^2 \quad \forall x$. Since for any convex function $h$ and fixed  point $x$, the Moreau smoothing  function  $h_\gamma(y;x) = h(y) +  \frac{1}{2 \gamma} \|y -x\|^2$ is
strongly convex in the first argument $y$ with strong convexity constant
$1/\gamma$, the following holds for all $x, y \in \rset^n$:
\begin{align}
\label{rel1_prox} 
& h(y)  \!+\!  \frac{1}{2 \gamma} \|y - x\|^2   \!\geq\!  h(\prox_{\gamma h}(x)) \!+\! \frac{1}{2 \gamma} \|\prox_{\gamma h}(x)
- x\|^2 \!+\! \frac{1}{2 \gamma} \|\prox_{\gamma h}(x) - y\|^2. 
\end{align}

\noindent The next theorem provides a descent property for the SPP
iteration and its proof is based on   the
previous  Moreau smoothing condition.

\begin{theorem}
\label{th:spp_basic} Let $f(\cdot, \xi)$ and $g(\cdot, \xi)$ be convex functions  that  admit easily computable proximal operators. Additionally, assume that the  stochastic bounded gradient condition from Assumption \ref{assumption1} holds and $g(\cdot, \xi)$ have bounded subgradients, that is there exists $B_{g}>0$ such that $ \|\nabla g(x, \xi)\| \leq B_g$ for all $x$ and $\xi$. Then, for any stepsize $\gamma_t>0$ we have the following recursive inequality for the SPP iteration:
\begin{align}
\label{spp_basic} & \Exp{\|x_{t+1} - \overline{x}_{t+1} \|^2} \leq \\
&  \Exp{\| x_t - \overline{x}_{t} \|^2}   - \gamma_t (2 - \gamma_t L) \Exp{ F(x_t) -
F(\overline{x}_{t})}  + \gamma_t^2 \left( B^2 + B_{g}^2 \right). \nonumber
\end{align}
\end{theorem}

\begin{proof}
Using \eqref{rel1_prox} for the function $h=f(\cdot,\xi_t)$ with $y
= \overline{x}_{t}$ and $x = x_t$ we get:
\[  f(\overline{x}_{t},\xi_t) + \frac{1}{2 \gamma_t} \|x_t - \overline{x}_{t}\|^2 \geq
f(x_{t+1/2},\xi_t) + \frac{1}{2 \gamma_t} \|x_{t+1/2} - x_t\|^2 +
\frac{1}{2 \gamma_t} \|x_{t+1/2} - \overline{x}_{t}\|^2. \] Using
again \eqref{rel1_prox} for  function $h=g(\cdot,\xi_t)$ with $y = \overline{x}_{t}$
and $x = x_{t+1/2}$ we get:
\[  g(\overline{x}_{t},\xi_t) + \frac{1}{2 \gamma_t} \|x_{t+1/2} - \overline{x}_{t}\|^2 \geq
g(x_{t+1},\xi_t) + \frac{1}{2 \gamma_t} \|x_{t+1} - x_{t+1/2}\|^2 +
\frac{1}{2 \gamma_t} \|x_{t+1} - \overline{x}_{t}\|^2. \] Adding the
previous two inequalities  and using the convexity of $f(\cdot,\xi_t)$
and $g(\cdot,\xi_t)$, we get:
\begin{align*}
& f(\overline{x}_{t},\xi_t) +  g(\overline{x}_{t},\xi_t) + \frac{1}{2
\gamma_t} \|x_t - \overline{x}_{t}\|^2 -
\frac{1}{2 \gamma_t} \|x_{t+1} - \overline{x}_{t}\|^2\\
& \geq f(x_{t+1/2},\xi_t) + g(x_{t+1},\xi_t) + \frac{1}{2 \gamma_t}
\|x_{t+1/2} - x_t\|^2 + \frac{1}{2 \gamma_t} \|x_{t+1} -
x_{t+1/2}\|^2
\\
& \geq  f(x_t,\xi_t) + g(x_t,\xi_t)  +  \langle \nabla f(x_t,\xi_t) +
\nabla g(x_t,\xi_t) , x_{t+1/2} - x_t \rangle 
\\
&\qquad + \frac{1}{2 \gamma_t} \|x_{t+1/2} -
x_t\|^2 + \langle \nabla g(x_t,\xi_t), x_{t+1}
- x_{t+1/2}  \rangle + \frac{1}{2 \gamma_t} \|x_{t+1} -
x_{t+1/2}\|^2\\
& \geq f(x_t,\xi_t) + g(x_t,\xi_t)  -\frac{\gamma_t}{2} \left( \| \nabla
f(x_t,\xi_t) + \nabla g(x_t,\xi_t) \|^2 + \|\nabla g(x_t,\xi_t)\|^2 \right),
\end{align*}
where in the last inequality we used that $\langle \alpha,z \rangle
 + \frac{1}{2\gamma} \|z\|^2 \geq - \frac{\gamma}{2} \|\alpha\|^2$ for all
$z$. Taking now  expectation and using that $\|x_{t+1} -
\overline{x}_{t}\|^2 \geq \|x_{t+1} - \overline{x}_{t+1}\|^2$ we get:
\begin{align*}
& \E_{\xi_{[t]}} [\|x_{t+1} - \overline{x}_{t+1} \|^2] \leq  \E_{\xi_{[t-1]}} [\| x_t - \overline{x}_{t} \|^2   - 2 \gamma_t  \E_{\xi_{[t-1]}}[ F(x_t) - F(\overline{x}_{t})]  \\ 
& \quad + \gamma_t^2 \E_{\xi_{[t]}} [ \| \nabla
F(x_t,\xi_t)   \|^2 + \| \nabla g(x_t,\xi_t) \|^2 ],
\end{align*}
where we used basic properties of the conditional expectation as in the proof of Theorem \ref{th:spg_basic}. Further, using the stochastic bounded gradient condition \eqref{as:main1_spg}  and  omitting the dependence  of expectation on $\xi_{[t]}$, we get  the statement of the theorem.  \qed
\end{proof}

\vspace{-0.4cm}

\subsection{Convergence rates for stochastic first order  methods}
\vspace{-0.1cm}
\noindent Let us denote $R_t^2 = \E_{\xi_{[t-1]}} [\|x_t -
\overline{x}_{t}\|^2]$. It follows that $R_0= \|x_0 -
\overline{x}_{0}\|^2$.  From Theorem \ref{th:spg_basic} for SPG and Theorem \ref{th:spp_basic} for SPP, we can derive various convergence rates for the stochastic first order (SFO) methods (i.e. SPG, SPP) depending on  the values taken by the constants $B, L$ and $\mu$ and also depending on the choice of the stepsize: constant or variable. Our convergence results recover, complement or extend the previous convergence rates for stochastic gradient descent (SGD) to more general  functions $f(\cdot,\xi)$  and  $g(\cdot,\xi)$.  To derive the convergence rates for the two SFO algorithms derived in the previous sections,  first we can notice that the descent relations from  Theorem \ref{th:spg_basic}  and Theorem \ref{th:spp_basic} can be written compactly as:
\begin{align}
\label{sfo_basic} 
& \Exp{\|x_{t+1} - \overline{x}_{t+1} \|^2} \leq   \Exp{\| x_t - \overline{x}_{t} \|^2}   - \gamma_t (2 - \gamma_t L) \Exp{ F(x_t) - F(\overline{x}_{t})}  + \gamma_t^2 {\cal B}^2, 
\end{align}
where 
\begin{equation*}
{\cal B}^2  \leq \left\{ \begin{array}{lr}
B^2 & \text{if}  \quad \;  \text{SPG} \\
B^2+B_g^2  & \;\; \text{if}  \quad \text{SPP}.
\end{array}\right.
\end{equation*}
Based on this common descent property  of our stochastic first order methods, for  the convex case (i.e. $\mu=0$) and constant stepsize we get the following convergence result (we make the convention $2/0 =\infty$):

\begin{theorem}[Sublinear convergence with constant stepsize]
\label{th:spg_mu0}
Let assumptions of Theorem \ref{th:spg_basic} for SPG and  Theorem \ref{th:spp_basic} for SPP  hold. Additionally, we assume constant stepsizes $\gamma_t \equiv \gamma \in (0,
2/L)$. Then, we  get the following relation for the average of
iterates $\hat{x}_t = \frac{1}{t} \sum_{j=0}^{t-1}x_j$  of SFO methods (i.e. SPG, SPP)  in the expected value function gap:
\begin{equation}
\label{eq-sfb-rate0}
\Exp{ F(\hat{x}_t) - F^*} \leq \frac{R_0^2}{ t \gamma(2 - \gamma L)}
+ \frac{\gamma {\cal B}^2}{(2-\gamma L)} \qquad \forall t \geq 1.
\end{equation}
Consequently, let $T$  the number of iterations of SFO be fixed such that $T {\cal B}^2 > R_{0}^2 L^2$ and $\gamma_t \equiv  \gamma = R_0/\sqrt{T {\cal B}^2}$, then we obtain ${\cal O}(1/\sqrt{T})$ convergence rate:
\[ \Exp{ F(\hat{x}_T) - F^*} \leq
\frac{2 R_0  {\cal B}  }{\sqrt{T}}.
\] 
If $ {\cal B}=0$ and $\gamma_t \equiv \gamma \in (0,
2/L)$, then we obtain ${\cal O}(1/T)$ convergence rate:
\begin{equation*}
\Exp{ F(\hat{x}_T) - F^*} \leq \frac{R_0^2}{ T \gamma(2 - \gamma L)}.
\end{equation*}
\end{theorem}

\begin{proof}
Adding the inequality \eqref{sfo_basic}  from $j=0$ to $j=t-1$ for
$\gamma_t=\gamma$, we get:
\begin{align*}
\gamma (2 - \gamma L) \sum_{j=0}^{t-1} \E \left[ F(x_j) -
F(\overline{x}_{j})\right] \leq \E [\|x_{0} - \overline{x}_{0}\|^2]
- \E[\|x_{t} - \overline{x}_{t}\|^2] + t \gamma^2 {\cal B}^2.
\end{align*}
Using Jensen's inequality for the convex  function $F$ and $\gamma \in (0, 2/L)$, we get:
\begin{align}
\label{eq-sfb-rate00}
\gamma (2 - \gamma L) t \E \left[ F(\hat{x}_t) - F^* \right]& \leq
\gamma (2 - \gamma L) t \E \left[ \sum_{j=0}^{t-1} \frac{1}{t}
F(x_j) - F^* \right] \nonumber\\
 &\leq \|x_{0} - \overline{x}_{0}\|^2 + \gamma^2
t{\cal B}^2,
\end{align}
which leads immediately to our first statement.  Fixing the number of iteration $T$ and minimizing the right hand side in \eqref{eq-sfb-rate0} w.r.t. $\gamma$ we obtain the best stepsize of the form ${\cal O}(1/\sqrt{T})$ and the corresponding convergence  rate ${\cal O}(1/\sqrt{T})$. However, the resulting expressions are cumbersome.  On the other hand, if
$T {\cal B}^2 > R_{0}^2 L^2$ and consider  the stepsize
$\gamma =   R_0/\sqrt{T {\cal B}^2}$ we get a simpler expression for the convergence rate. Indeed,  for these choices we have that $2 - \gamma L \geq 1$ and using this  in the  inequality \eqref{eq-sfb-rate00}, we obtain $ \Exp{ F(\hat{x}_T) - F^*} \leq \frac{2 R_0 {\cal B}}{ \sqrt{T}}$.  Convergence rate ${\cal O}(1/T)$ for  ${\cal B} =0$ follows immediately from \eqref{eq-sfb-rate00}.  \qed
\end{proof}

\begin{theorem}[Sublinear convergence with variable stepsize]
\label{th:spg_mu0var}
Let assumptions of Theorem \ref{th:spg_basic} for SPG and  Theorem \ref{th:spp_basic} for SPP  hold. Additionally, we assume $L >0$, variable stepsize $\gamma_t = \gamma_0/\sqrt{t}$ with $\gamma_0 = 1/L$  and the distance of the iterates to the optimal set is bounded almost surely, i.e. $\Exp{\| x_t - \overline{x}_{t} \|^2}  \leq R^2$ for all $t \geq 0$. Then, we  get the following  rate for the average of
iterates $\hat{x}_t = \frac{1}{t} \sum_{j=0}^{t-1}x_j$  of SFO in the
expected value function gap:
\begin{equation}
\label{eq-sfb-rate0var}
\Exp{ F(\hat{x}_t) - F^*} \leq   \frac{1}{ \sqrt{t}} (R^2 L +  \frac{2 {\cal B}^2}{L})  \qquad \forall t \geq 1.
\end{equation}
\end{theorem}

\begin{proof}
Since $\gamma_t = \gamma_0/\sqrt{t}$ with $\gamma_0 = 1/L$ it follows that $\gamma_t  (2 - \gamma_t L) \geq \gamma_t$, which combined with \eqref{sfo_basic} yields: 
$$ \Exp{ F(x_t) - F^*} \leq \frac{1}{\gamma_t } \Exp{x_t - \overline{x}_{t} \|^2} -  \frac{1}{ \gamma_t}   \Exp{\| \|x_{t+1} - \overline{x}_{t+1}  \|^2} + \gamma_t {\cal B}^2  $$
Summing up from $j = 0$ to $t-1$ and using that $\gamma_t$ is a nonincreasing sequence, we have:
\begin{align*} 
& \sum_{j = 0}^{t-1}  \Exp{ F(x_j) - F^*}  \leq \frac{1}{\gamma_0} \|x_0 - \overline{x}_{0} \|^2 +  \sum_{j= 0}^{t-2} (\frac{1}{\gamma_{j+1}} - \frac{1}{\gamma_{j}}) \Exp{\| x_{j+1} - \overline{x}_{j+1}\|^2} \\ 
& \qquad -  \frac{1}{ \gamma_{t-1}}\Exp{\| x_{t} - \overline{x}_{t} \|^2} + {\cal B}^2  \sum_{j=0}^{t-1} \gamma_{j}  \\
& \leq  \frac{1}{\gamma_0} \|x_0 - \overline{x}_{0}  \|^2 + \sum_{j= 0}^{t-2} (\frac{1}{\gamma_{j+1}} - \frac{1}{\gamma_{j}}) \Exp{\| x_{j+1} - \overline{x}_{j+1}\|^2}  + {\cal B}^2  \sum_{j=0}^{t-1} \gamma_{j} \\
& \leq  \frac{R^2}{\gamma_0}  + \sum_{j= 0}^{t-2} (\frac{1}{\gamma_{j+1}} - \frac{1}{\gamma_{j}}) R^2  + {\cal B}^2  \sum_{j=0}^{t-1} \gamma_{j} =   \frac{R^2}{\gamma_0} + (\frac{1}{\gamma_{t-1}} - \frac{1}{\gamma_{0}}) R^2 + {\cal B}^2  \sum_{j=0}^{t-1} \gamma_{j} \\
& =  \frac{R^2}{\gamma_{t-1}} + {\cal B}^2  \sum_{j=0}^{t-1} \gamma_{j}.
\end{align*}
Using now the Jensen's inequality for the convex  function $F$, we get:
\[  \Exp{ F(\hat{x}_t) - F^*} \leq   \frac{R^2}{t \gamma_{t-1}} + \frac{{\cal B}^2}{t}  \sum_{j=0}^{t-1} \gamma_{j}  \leq \frac{R^2 L}{ \sqrt{t}} +  \frac{2 {\cal B}^2}{L \sqrt{t}},  \]
which proves our statement.  \qed
\end{proof}

\noindent If we assume additionally some quadratic functional growth condition (i.e. $\mu >0$) and run the two SFO methods with constant stepsize, the next theorem proves that  we can achieve linear convergence to a noise dominated region proportional to the stepsize and in some cases even pure linear convergence provided that ${\cal B}=0$.

\begin{theorem} [Sufficient conditions for linear convergence of SFO]
\label{th:spg_muL} Let assumptions of Theorem \ref{th:spg_basic} for SPG and  Theorem \ref{th:spp_basic} for SPP hold. Additionally, we  assume that the quadratic functional growth condition from  Assumption \ref{assumption2} holds  and we choose constant stepsize  $\gamma_t = \gamma \in (0,2/L)$. Then, we get the following upper bounds for the square distance of the SFO iterates to the solution set:
\begin{equation*}
\Exp{\|x_{t} - \overline{x}_{t} \|^2} \!\leq\!\! \left\{ \!\begin{array}{lr}
\left| 1 - \mu \gamma + \frac{\mu L \gamma^2}{2} \right |^t  \!\!R_0^2 +\!
\frac{{\cal B}^2 \gamma^2}{1- |1-\mu \gamma  + \mu L \gamma^2/2|}  \!\!\!\!\!\!\!\! & \text{if}  \; 
1 \!-\! \mu \gamma + \frac{\mu L \gamma^2}{2} \!>\! -1 \\
{\cal B}^2 \gamma^2  & \text{if}  \quad  1 - \mu \gamma + \frac{\mu L
\gamma^2}{2} \leq -1.
\end{array}\right.
\end{equation*}
Moreover, if the condition $\mu < 4L$ holds or  the constant stepsize satisfies $\gamma < \min\left({2}/{L}, {2}/{\mu} \right)$, then  $|  1 - \mu \gamma + \frac{\mu L\gamma^2}{2}| < 1$ and therefore we get linear converge to a noise dominated region with radius  proportional to the square of the stepsize $\gamma$:
\[ \Exp{\|x_{t} - \overline{x}_{t} \|^2} \leq  \left| 1 - \mu \gamma + \frac{\mu L \gamma^2}{2} \right|^t  R_0^2 +
\frac{{\cal B}^2 \gamma^2}{1- |1-\mu \gamma  + \mu L \gamma^2/2|}. \]
Consequently, if ${\cal B}=0$, then we get pure linear convergence:
\begin{equation*}
\Exp{\|x_{t} - \overline{x}_{t} \|^2} \leq \left| 1 - \mu \gamma + \frac{\mu L \gamma^2}{2} \right |^t  R_0^2.
\end{equation*}
\end{theorem}

\begin{proof}
From \eqref{sfo_basic}  and Assumption \ref{assumption2}  it follows that:
\begin{align*}
& \E_{\xi_{[t]}} [\|x_{t+1}  - \overline{x}_{t+1}\|^2] \leq
\E_{\xi_{[t-1]}}[\|x_{t} - \overline{x}_{t}\|^2]  -  \gamma_t (2 -
\gamma_t L) \E_{\xi_{[t-1]}}
\left[F(x_t) - F(\overline{x}_{t})\right]  \\
& + {\cal B}^2 \gamma_{t}^2  \overset{\eqref{as:strong_spg}}{\leq} \E_{\xi_{[t-1]}}[\|x_{t} -
\overline{x}_{t}\|^2]  - \frac{\mu \gamma_t (2 - \gamma_t L)}{2}
\E_{\xi_{[t-1]}} [\|x_{t} - \overline{x}_{t}\|^2] +  {\cal B}^2
\gamma_{t}^2,
\end{align*}
provided that $\gamma_t \leq 2/L$.  In conclusion, we obtain the
following recursion:
\begin{align}
\label{recursive1_spg} \E_{\xi_{[t]}} [\|x_{t+1} -
\overline{x}_{t+1}\|^2] & \leq   \left(1 - \mu \gamma_t + \frac{\mu
L \gamma_t^2}{2} \right) \E_{\xi_{[t-1]}} [\|x_{t} -
\overline{x}_{t}\|^2] +  {\cal B}^2 \gamma_{t}^2.
\end{align}
For constant stepsize $\gamma_t = \gamma < 2/L$ it follows that $ 1 - \mu \gamma
+ \mu L \gamma^2/2 < 1$.  Then, combining inequality \eqref{recursive1_spg} with
condition $1 - \mu \gamma + \mu L
\gamma^2/2 > -1$, we get:
\begin{align*}
\E [\|x_{t} - \overline{x}_{t}\|^2] & \leq   |1 - \mu \gamma +
\frac{\mu L \gamma^2}{2} |^t \|x_{0} - \overline{x}_{0}\|^2 +
{\cal B}^2 \gamma^2 \sum_{j=0}^{t-1} |1 - \mu \gamma +
\frac{\mu L \gamma^2}{2}|^j \\
& \leq |1 - \mu \gamma + \frac{\mu L \gamma^2}{2} |^t
\|x_{0} - \overline{x}_{0}\|^2 +    \frac{{\cal B}^2 \gamma^2}{ 1- |1-\mu \gamma
+ \mu L \gamma^2/2|},
\end{align*}
where in the second inequality we used that $\sum_{j=0}^{t-1} c^j
\leq \sum_{j=0}^{\infty} c^j = \frac{1}{1-c}$ for any $|c|<1$.
Otherwise, if $1 - \mu \gamma + \mu L \gamma^2/2 \leq -1$, then from
\eqref{recursive1_spg} it follows that:
\begin{align*} \E [\|x_{t} - \overline{x}_{t}\|^2] &
\leq {\cal B}^2 \gamma^2.
\end{align*}
Moreover, if $\mu < 4L$, then the relation $1 - \mu \gamma + \mu L
\gamma^2/2 > -1$ always holds, regardless of the value of $\gamma$. Thus, if $\mu < 4L$ and  the constant stepsize is chosen such $\gamma < \frac{2}{L}$,  then $|1 - \mu \gamma + \mu L \gamma^2| < 1$. Similarly, if  the  constant stepsize satisfies  $\gamma < \min(2/L, 2/\mu)$, then again we have $0< 1 - \mu \gamma + \mu L \gamma^2 < 1$. Therefore, these two previous conditions guarantee linear convergence  to a noise dominated region whose radius is proportional to the square of the stepsize $\gamma$.    Finally, the inequality \eqref{recursive1_spg} with ${\cal B}=0$ and $|1 - \mu \gamma + \mu L \gamma^2| < 1$  proves  linear convergence for $\Exp{\|x_{t} - \overline{x}_{t} \|^2}$, with $x_t$ generated by SFO methods.  \qed
\end{proof}

\begin{remark}
\noindent Note that if $f$ is strongly convex and with Lipschitz continuous gradient,   then the condition $\mu < 4L$  holds. Indeed, assume that the $\mu$ strongly convex function  $f(\cdot,\xi)$ has $L_f$ Lipschitz continuous gradient and the regularization function $g$ has bounded subgradients. Then, we have:
\[  \|\nabla f(x, \xi) - \nabla f(x^*,\xi)\| \leq L_f \|x - x^*\| \quad \text{and} \quad
\|\nabla g(x, \xi) (x)\| \leq B_g \quad \forall x \in \textnormal{dom} \ F, \]
where $x^*$ is the unique solution of \eqref{eq_problem}.   In this case it follows that:
\begin{align*}
& \|\nabla f(x,\xi) + \nabla g(x, \xi)  \|  \leq  \|\nabla f(x,\xi) - \nabla
f(x^*, \xi) \| + \|\nabla f(x^*,\xi) + \nabla g(x, \xi)   \| \\
& \!\leq\! L_f \| x - x^*\| + \|\nabla f(x^*,\xi)\|
+ \| \nabla g(x, \xi)   \|  \!\leq\!   L_f \| x - x^*\| \!+\! (\|\nabla f(x^*,\xi)\| + B_g).
\end{align*}
Now, since $f$ is strongly convex, it follows that $F(x) - F^* \geq \frac{\mu}{2} \| x - x^*\|^2$, which inserted in the previous inequality, yields: 
\begin{align*}
\|\nabla f(x,\xi) + \nabla g(x, \xi)  \| & \leq  L_f \left(\frac{2(F(x) -
F^*)}{\mu} \right)^{1/2} + (\|\nabla f(x^*,\xi)\| + B_g).
\end{align*}
Squaring the fist and the last term in the previous relation and taking expectation  we get  that the stochastic bounded gradient condition  \eqref{as:main1_spg} holds with:
 $ L= \frac{4 L_f^2}{\mu}  >  \frac{\mu}{4}$, 
since we  always have the relation $L_f \geq \mu$  between the Lipschitz and strong convexity constants of a convex function.  \qed
\end{remark}

\noindent Theorem \ref{th:spg_muL} states that SFO methods achieve linear convergence when $B=0$ in \eqref{as:main1_spg} and in this case we say that the objective function $F$ satisfies a \textit{strong} stochastic bounded gradient condition.   Next theorem derives also  necessary conditions to achieve linear convergence for SFO methods with constant stepsize. Paper  \cite{NecNes:15} proves that the class of  objective functions having Lipschitz continous gradient and satisfying a  quadratic functional growth is the largest one for which deterministic gradient method converges linearly.  Using a similar reasoning as in  \cite{NecNes:15},  we expect  additional necessary conditions for linear convergence of SFO methods as described in the next result.   

\begin{theorem}  [Necessary conditions for linear convergence of SFO]
\label{th_spg_nec}
Assume  that $g \equiv 0$ and  $f$ satisfies a quadratic functional growth condition (Assumption \ref{assumption2}) and has Lipschitz continuous gradient.  Assume further that there exists $q < 1$ such that the SFO iterates with constant stepsize $\gamma>0$  converge linearly, i.e. $\E_{\xi_t}[\|x_{t+1} - \overline{x}_{t+1}\|^2 | \xi_{[t-1]}] \leq q \|x_{t} - \overline{x}_{t}\|^2$ for all $t \geq 0$,  and the stochastic process $(x_t)_{t \geq 0}$ of SFO yields the same  projection onto the optimal set, i.e. $\overline{x}_{t}  = \overline{x}_{0}$ for all $t \geq 0$ almost surely. Then, the stochastic bounded  gradient condition from  Assumption \eqref{assumption1} holds with $B \equiv 0$, or equivalently, all the partial functions $f(\cdot,\xi)$ must have the same minimizer.
\end{theorem}

\begin{proof}
Note that for $g=0$ the SPG iterates coincide with the SGD updates: $x_{t+1} = x_t - \gamma \nabla f(x_t, \xi_t)$. Then, using that the stochastic process $(x_t)_{t \geq 0}$ yields  the same  projection onto the optimal set, i.e. $\overline{x}_{t}  = \overline{x}_{0}$ for all $t \geq 0$, we get:
\begin{align*}
& \gamma^2 \|\nabla f(x_t, \xi_t)\|^2  = \| x_{t+1} - x_t \|^2 \\ 
& \leq  2 \|x_{t+1} - \overline{x}_{t} \|^2 + 2 \|x_{t} - \overline{x}_{t} \|^2 = \ 2 \|x_{t+1} - \overline{x}_{0} \|^2 + 2 \|x_{t} - \overline{x}_{0} \|^2.
\end{align*}
Taking expectation with respect to $\xi_{t}$ and using $\E_{\xi_t}[\|x_{t+1} - \overline{x}_{t+1}\|^2| \xi_{[t-1]}] \leq q \|x_{t} - \overline{x}_{t}\|^2$ and the quadratic functional growth condition \eqref{as:strong_spg} on $f$, we obtain:
\begin{align}
\label{rel1} 
\!\!  \E_{\xi_t}[ \|\nabla f(x_t,\xi_t) \|^2 | \xi_{[t-1]} ] \!\leq\!  \frac{2(1 \!+\! q)}{ \gamma^2} \|x_{t} \!-\! \overline{x}_{0}\|^2 \!\leq\! \frac{4(1 \!+\! q)}{\mu \gamma^2} (f(x_t) \!-\!  f(\overline{x}_{0})). 
\end{align}

\noindent Similarly, for SPP we have from the definition of the proximity operator:
$$
x_t - x_{t+1} = \gamma \nabla f(x_{t+1},\xi_t)= \gamma \nabla f(x_t,\xi_t) +\gamma \nabla f(x_{t+1},\xi_t)-\gamma \nabla f(x_t,\xi_t).
$$
Therefore,  from the $L_f$  Lipschitz continuity  of $\nabla f(\cdot,\xi)$, we get
\begin{align*}
& \gamma^2 \| \nabla f(x_t,\xi_t) \|^2  =  \|x_t- x_{t+1} \|^2 + 2\gamma \langle x_t - x_{t+1},  \nabla f(x_t,\xi_t)- \nabla f(x_{t+1},\xi_t) \rangle \\
& \quad +\gamma^2 \| \nabla f(x_{t+1},\xi_t) - \nabla f(x_t,\xi_t) \|^2 \leq (1 + 2 \gamma L_f + \gamma^2 L_f) \| x_t -x_{t+1}\|^2 \\
& \leq 2(1 + 2 \gamma L_f + \gamma^2 L_f)( \|x_{t+1} - \overline{x}_t\|^2 +\| x_t - \overline{x}_t\|^2.
\end{align*}
Now, taking expectation on the  both sides of above inequality, using that the stochastic process $(x_t)_{t \geq 0}$ produces the same  projection onto the optimal set, i.e. $\overline{x}_{t}  = \overline{x}_{0}$ for all $t \geq 0$, and  $\E_{\xi_t}[\|x_{t+1} - \overline{x}_{t+1} \|^2 | \xi_{[t-1]}] \leq q \|x_t -\overline{x}_t\|$, we get:
\begin{align}
\label{rel1spp}
& \gamma^2 \E_{\xi_t}[ \| \nabla f(x_t,\xi_t) \|^2 | \xi_{[t-1]}] \leq 2(1 + 2 \gamma L_f + \gamma^2 L_f)(1+q) \| x_t - \overline{x}_0\|^2  \nonumber \\
&\leq (4/\mu)(1 + 2 \gamma L_f + \gamma^2 L_f)(1+q) (f(x_t) - f(\overline{x}_0)),
\end{align}
where in  the last inequality we used the  quadratic functional growth condition \eqref{as:strong_spg}  on $f$.  On the other hand, if $f(\cdot,\xi)$ has $L_f$ Lipschitz continuous gradient,  then for any $x \in \text{dom} \ F$ we have: 
\begin{align*}
&\|\nabla f(x,\xi)  \|^2  \leq 2  \|\nabla f(x,\xi) - \nabla
f(\overline{x}_{0}, \xi) \|^2 + 2 \|\nabla f(\overline{x}_{0},\xi)  \|^2 \\
& \leq 2 L_f^2 \| x - \overline{x}_{0} \| ^2+ 2 \|\nabla f(\overline{x}_{0},\xi)\|^2  \leq   \frac{4 L_f^2}{\mu} (f(x) - f(\overline{x}_{0}))  + 2 \|\nabla f(\overline{x}_{0},\xi)\|^2, 
\end{align*} 
where in the last inequality we used again  the quadratic functional growth condition \eqref{as:strong_spg} on $f$.  Taking now expectation with respect to $\xi$, we obtain:
\begin{align}
\label{rel2}
 \E_{\xi}[ \|\nabla f(x,\xi) \|^2 \leq   \frac{4 L_f^2}{\mu} (f(x) - f(\overline{x}_{0}))   + 2  \E_{\xi}[ \|\nabla f(\overline{x}_{0},\xi)\|^2 ]  \; \forall x \!\in\! \text{dom} \, F.
\end{align} 
Combining \eqref{rel1} or \eqref{rel1spp}  with \eqref{rel2}, it follows that $B = 2  \E_{\xi}[ \|\nabla f(\overline{x}_{0},\xi)\|^2 ] = 0$, i.e. the \textit{strong} stochastic bounded gradient condition holds in \eqref{as:main1_spg}. This shows that if SFO iterates converge linearly on the class of objective functions having Lipschitz  gradient and  quadratic functional growth, then the partial functions $f(\cdot,\xi)$ must have the same minimizer,  i.e.  $\nabla f(\overline{x}_{0},\xi) = 0$ alsmost surely.~\qed
\end{proof}

\noindent In the next  theorem we show that SFO algorithms based on a hybrid  strategy consisting  of using constant stepsize $\gamma=1/L$ at the beginning and then at some well-defined iteration switching to a  variable stepsize $\gamma={\cal O}(1/t)$ has a sublinear convergent behavior of order ${\cal O}(1/t)$.

\begin{theorem} [Sublinear convergence with constant-variable stepsize]
\label{th:spg_mu_var1} Let Assumptions  \ref{assumption1}  and Assumption  \ref{assumption2}  hold, and choose variable stepsize $\gamma_t =\min \left(\frac{1}{L}, \frac{c}{(t+1)} \right)$,  for some fixed  $c > 0$. Set $d = c^2 {\cal B}^2$ and $t_0= \lfloor cL \rfloor$. Then, we get the following sublinear convergence rate for the square distance of the SFO iterates to the solution set:
\begin{align*}
& \E[\|x_{t} - \overline{x}_{t}\|^2 ] \\
& \quad \leq
\left\{\begin{array}{ll}
\frac{ 2t_0R_{t_0}^2 + 2d(1+ \log((t+1)/(t_0+1)))}{t+1}& \text{if}  \quad c\mu=2 \; \& \; t > t_0,\\
\left(  R_{t_0}^2(1+t_0)^{0.5c\mu} + \frac{2d(2 - 0.5 c \mu)(t_0+1)^{0.5c\mu}}{1-0.5c\mu}\right) \frac{1}{(t+1)^{0.5 c\mu}}
& \text{if} \quad c\mu<2 \; \& \; t > t_0,\\
\left((t_0+1)R_{t_0}^2 + \frac{d}{0.5 c\mu-1}\right) \frac{1}{t+1} &\text{if} \quad c\mu>2 \; \& \; t > t_0\\
\max \left( |1 - \frac{\mu}{2L}|^t  R_0^2 +
\frac{{\cal B}^2}{L^2- |L^2- \mu L/2|}, {\cal B}^2/L^2 \right) & \text{if} \quad  t \leq t_0.
\end{array}\right.
\end{align*}
\end{theorem}

\begin{proof}
Since for $t \leq t_0$ the stepsize is $\gamma_t=1/L$, it follows from Theorem \ref{th:spg_muL}  that $\E_{\xi_{[t_0-1]}} [\|x_{t_0} - \overline{x}_{t_0}\|^2]$ is
bounded for example by:
\[  \E_{\xi_{[t_0-1]}} [\|x_{t_0} - \overline{x}_{t_0}\|^2] \leq \max \left( |1 -  \frac{\mu}{2L}|^{t_0}  R_0^2 + \frac{{\cal B}^2}{L^2- |L^2- \mu L/2|}, {\cal B}^2/L^2 \right). \]
Moreover,  for our hybrid choice of the  stepsizes $\gamma_t =\min\left(
\frac{1}{L}, \frac{c}{(t+1)} \right)$ we always have $2-\gamma_t L \geq 1$
and consequently  combining \eqref{spg_basic}  with Assumption \ref{assumption2}  it follows that:
\begin{align*}
& \E_{\xi_{[t]}} [\|x_{t+1}  - \overline{x}_{t+1}\|^2] \\
& \leq \E_{\xi_{[t-1]}}[\|x_{t} - \overline{x}_{t}\|^2]  -  \gamma_t (2 -
\gamma_t L) \E_{\xi_{[t-1]}}
\left[F(x_t) - F(\overline{x}_{t})\right]  + {\cal B}^2 \gamma_{t}^2 \\
& \overset{\eqref{as:strong_spg}}{\leq} \E_{\xi_{[t-1]}}[\|x_{t} -
\overline{x}_{t}\|^2]  - \frac{\mu \gamma_t (2 - \gamma_t L)}{2}
\E_{\xi_{[t-1]}} [\|x_{t} - \overline{x}_{t}\|^2] +  {\cal B}^2
\gamma_{t}^2\\
&  \leq  \E_{\xi_{[t-1]}}[\|x_{t} -
\overline{x}_{t}\|^2]  - \frac{\mu \gamma_t }{2}
\E_{\xi_{[t-1]}} [\|x_{t} - \overline{x}_{t}\|^2] +  {\cal B}^2
\gamma_{t}^2\\ 
& =  \left(1 - \frac{c\mu}{2(t+1)} \right) \E_{\xi_{[t-1]}} [\|x_{t} -
\overline{x}_{t}\|^2]  + \frac{c^2 {\cal B}^2}{(t+1)^{2}} \quad \forall t
\geq t_0.
\end{align*}
Therefore,  we obtain the recurrence $r_{t+1} \leq  \left(1 - \frac{c\mu}{2(t+1)} \right) r_t + \frac{c^2 {\cal B}^2}{(t+1)^{2}} $ and then using Lemma \ref{lem_polyak1}  we get our statements. \qed 
\end{proof}

\noindent From previous theorem we observe that the best convergence rate ${\cal O}(1/t)$ is obtained when the constant  $c$ is proportional to the inverse of the strong convexity constant $\mu$, and thus the switching time $t_0$ has to be about two times the condition number $L/\mu$ of the optimization problem \eqref{eq_problem}, i.e.:
\[  c \approx \frac{2}{\mu} \quad \text{and} \quad t_0 \approx  2 \frac{L}{\mu}.   \]

\begin{remark}
The previous convergence results for SFO algorithms can be easily  extended  to a more general and robust variable stepsize of the form:
\[ \gamma_t =\min \left( \frac{1}{L}, \frac{c}{(t+1)^\alpha} \right), \]
for some fixed $c > 0$ and $\alpha \in (0,1)$. In this case we can  use Lemma \ref{lem_polyak2} (see also \cite{MouBac:11,PatNec:17}) for the  recurrence:
\[  r_{t+1} \leq \left( 1  - \frac{c \mu}{2(t+1)^\alpha}\right) r_t + \frac{c^2 {\cal B}^2}{(t+1)^{2\alpha}} \qquad \forall t \geq t_0,  \]
to obtain sublinear convergence rate of order ${\cal O}(1/t^\alpha)$ for the square distance of the SFO iterates to the solution set. \qed
\end{remark}

\vspace{-0.4cm}

\subsection{Stochastic gradient descent revisited}
\vspace{-0.2cm}
\noindent   Note that our framework based on stochastic gradient
mapping and Moreau mapping allows us to generalize the existing results \cite{NedBer:00,SchRou:13}  for stochastic gradient descent (SGD), i.e. when $g$ is the indicator function of a simple convex set $C$, to the general case of stochastic first order (SFO) algorithms that are able to deal with any convex family of functions $g(\cdot,\xi)$ that admit a tractable proximal operator. In particular, when considering the stochastic  convex problem $\min_{x \in C} \Exp{ f(x,\xi)},$ we recover the basic results for the  SGD   algorithm:
\begin{equation*}
(SGD): \qquad  x_{t+1} = \Pi_{C}(x_t - \gamma_t \nabla
f(x_t,\xi_t)).
\end{equation*}
More precisely, it follows  that the conclusions of Theorems
\ref{th:spg_basic},  \ref{th:spg_mu0}  and \ref{th:spg_mu0var} are still valid
provided that Assumption \ref{assumption1} is replaced with:
\begin{equation}
\label{e:main1} B^2 + L(f(x) - f(\overline{x})) \geq
\Exp{\|\nabla f(x,\xi)\|^2} \quad \forall x \in C.
\end{equation}
Moreover,  Theorems \ref{th:spg_muL} and \ref{th_spg_nec}  hold, when Assumption \ref{assumption2} is
replaced with:
\begin{equation}
\label{e:strong} f(x) - f(\overline{x}) \geq  \frac{\mu}{2} \|x -
\overline{x}\|^2 \quad \forall x \in C.
\end{equation}
In particular,   from Theorem \ref{th:spg_muL} under the conditions  \eqref{e:main1} with $B=0$ and \eqref{e:strong}  we get
linear convergence for SGD. Note that our strong stochastic bounded gradient condition, i.e.  \eqref{e:main1} with $B=0$,  may be more general than
the one imposed in \cite{SchRou:13}   to obtain linear convergence for SGD . More precisely,  \cite{SchRou:13} considers $C = \rset^n$ and the  strongly convex function $f$ with Lipschitz gradient. Besides these assumptions \cite{SchRou:13} also
imposes a strong growth condition:
\[  L_\text{sg} \E[\| \nabla f(x,\xi) \|^2] \leq  \| \E[\nabla f(x,\xi)]
\|^2 \quad \forall x \in \rset^n, \] for some $L_\text{sg} >0$.  Clearly,
if the strong growth condition holds and $f$ has gradient Lipschitz with constant $L_f$, then \eqref{e:main1} also  holds with  $B=0$, since for the
unconstrained case,   requirement that $f$ has
gradient Lipschitz leads to:
\begin{align*}
f(x) & \geq  f(\overline{x}) + \langle \nabla f(\overline{x}), x - \overline{x} \rangle + \frac{1}{2L} \|\nabla f(x) - \nabla f(\overline{x}) \|^2 \\
& \overset{\nabla f(\overline{x})=0}{=} f(\overline{x})  +
\frac{1}{2L} \|\Exp{\nabla f(x,\xi)} \|^2  \geq f(\overline{x})  + \frac{L_\text{sg}}{2L} \Exp{ \|\nabla
f(x,\xi)\|^2}.
\end{align*}

\noindent Note that the conditions \eqref{e:main1} with $B=0$  and
\eqref{e:strong}  are satisfied  in some applications. For
example, consider the problem of finding a point in the (in)finite
intersection of a family of simple closed convex sets, i.e.  $\text{find} \quad  x \in C = \cap_{\xi \in \Omega} C_{\xi}$.  Assume  that linear regularity holds for the sets $(C_{\xi})_{\xi \in \Omega}$, i.e.
there exists  finite constant $\kappa >0$ such that:
\begin{align}
\label{linreg}  \kappa \, \| x - \Pi_{C}(x) \|^2 \leq  \Exp{\| x -
\Pi_{C_\xi}(x)\|^2} \quad \forall x \in \rset^n.
\end{align}
Typical examples of sets satisfying linear regularity are polyhedra
sets or convex sets with  nonempty interior \cite{NecRic:16}. For
example,  when solving linear systems $Ax=b$ we can take $C_\xi=\{x:
a_\xi^Tx = b_\xi \}$, where $a_\xi$ is the $\xi$th row of matrix  $A
\in \rset^{m \times n}$. In  this case,  $\kappa =m/ \lambda_{\min}^\text{nz}(A^TA)$,  provided that we consider a  uniform probability
distribution on $\Omega=\{1,\cdots,m\}$,   where
$\lambda_{\min}^\text{nz}(A^TA)$ denotes the smallest nonzero eigenvalue
of $A^TA$. Then, under linear regularity the convex feasibility
problem can be equivalently written as a stochastic optimization
problem, see also \cite{NecRic:16}:

\vspace{-0.2cm}

\[ \min_{x \in \rset^m} \Exp{f(x,\xi)} \quad \left( \eqdef
\frac{1}{2} \Exp{\| x - \Pi_{C_\xi}(x)\|^2} \right).  \] 

\vspace{-0.1cm}

\noindent Then, if $C
\not = \emptyset$, it follows immediately that $f(x) =
\Exp{f(x,\xi)}$ satisfies condition \eqref{e:main1} with $L=2$ and
$B=0$, since the following relations hold: $2 f(x) =  \Exp{ \|\nabla f(x,\xi)\|^2}$ and  $f(\overline{x}) =0$.  Moreover, it is easy to see that the linear regularity
\eqref{linreg} can be written  as    \eqref{e:strong}  with $\mu = \kappa$.  In conclusion, from the results of the previous section we have that SGD or equivalently  the  stochastic alternating projection algorithm  $x_{t+1} = x_t -
\gamma_t (x_t - \Pi_{C_{\xi_t}}(x_t))$  converges linearly. Note
that for linear systems $Ax=b$ the previous alternating projection
scheme with $\gamma_t=1$ becomes Kaczmarz algorithm.  Similar linear
convergence rates for the alternating projection has been derived
in \cite{NecRic:16} and those results were the main motivation for this paper, that is extending the linear rates for stochastic alternating projection  to more general stochastic problems and algorithms.

\vspace{-0.3cm}

\section{Restarted stochastic first order methods}
\vspace{-0.2cm}
\noindent From previous sections we observe that linear convergence can be achieved by stochastic  first order (SFO) methods (i.e. SPG and SPP) provided that $B =0$ in Assumption \ref{assumption1}  and Assumption  \ref{assumption2}. However, this condition restricts the class of functions for which we can achieve linear convergence for SFO.  In this section  we propose  a restarting variant of SFO and prove that this new method can achieve linear convergence for classes of functions for which $B \not = 0$. This restarting variant consists of running the SPG/SPP algorithms (as a routine) for multiple times (epochs) and restarting it each time after a certain
number of iterations. In each epoch $t$, the SPG/SPP scheme  runs
for an estimated number of iterations $K_t$. More explicitly,  the
restarted stochastic first order (R-SFO) scheme has the following
iteration:
\begin{algorithm}
Let $x_{0,0} \in \text{dom}\, F$. For $t \geq 1$ compute stepsize
$\gamma_t$ and number of  iterations $K_t$. Iterate SPG/SPP
algorithm for $K_{t}$ iterations with constant stepsize $\gamma_t$
and starting from $x_{K_{t-1},t-1}$. Set $x_{K_t,t}$ the average of iterates.
\end{algorithm}

\noindent In the next theorem we derive the convergence rate of this
restarted scheme.

\begin{theorem}\label{th_rspg_linear}
Assume that  there exist positive constant $B$ and $B_g$ such that:
\begin{align}
\label{rspg_b}  \Exp{\| \nabla f(x, \xi) + \nabla g(x, \xi)\|^2} \leq B^2 \;
\text{and} \; \|\nabla g(x, \xi)\|^2 \leq B_g^2 \quad \forall \; x \in
\text{dom} \ F
\end{align}
and that $F$ satisfies a $\nu$ functional growth condition on any segment $[x, \; \bar{x}]$:
\begin{align}
\label{rspg_sm} F(x) - F(\overline{x}) \geq  \mu \|x -
\overline{x}\|^\nu \qquad \forall x \in \text{dom} \ F \quad \text{with} \quad \nu \in [1, \; 2].
\end{align}
Also let $\gamma_0 > 0$, ${\cal B}^2 = B^2 + B_g^2$ and
$\{x_{K_t,t}\}_{ t \ge 0}$ be generated by R-SFO with stepsize
$\gamma_t = \frac{\epsilon_{t-1}}{2 {\cal B}^2}$ and $K_{t} =
\left\lceil \frac{4
{\cal B}^2}{\mu^{2/\nu} \epsilon_{t-1}^{2 - 2/\nu}} \right\rceil$, where
$\epsilon_t = \frac{\epsilon_{t-1}}{2}$ and $\epsilon_0 \geq  F(x_{0,0})
- F^*$. If, for a given accuracy $\epsilon> 0 $, we perform  $ T = \left\lceil
\log\left(\frac{\epsilon_0}{\epsilon}\right) \right\rceil $ epochs, then
after a total number of iterations of the SPG/SPP schemes that is
bounded~by:
\[  \left\lceil \frac{4 {\cal B}^2}{\mu^{2/\nu} \epsilon^{2 - 2/\nu}} \right\rceil  \left\lceil \log\left(\frac{\epsilon_0}{\epsilon}\right)
\right\rceil, \qquad \text{we have} \qquad \mathbb{E}[F(x_{K_T,T})]
- F^* \le \epsilon. \]
\end{theorem}

\begin{proof}
For any $t \ge 1$, consider the $t$th epoch of the  R-SFO method.
Let the sequence $\{\epsilon_t\}_{ t\ge 0}$, with $\epsilon_0 \geq
F(x_{0,0}) - F^*$, satisfy:
\begin{align}
\label{th_rspg_linear_rel1} \mu \norm{x_{K_{t-1},t-1} -
\overline{x}_{K_{t-1},t-1}}^\nu \le \mathbb{E}[F(x_{K_{t-1},t-1})] - F^*
\le \epsilon_{t-1}.
\end{align}
Also consider the constant  stepsize $ \gamma_t  =
\frac{\epsilon_{t-1}}{2{\cal B}^2}$ and perform $K_{t} \ge \frac{4
{\cal B}^2}{\mu^{2/\nu} \epsilon_{t-1}^{2 - 2/\nu}}$  number of inner iterations of SPG/SPP. Since  \eqref{rspg_b} holds, then we can use Theorem \ref{th:spg_mu0}  to conclude that:
\begin{align*}
\mathbb{E}[F(x_{K_{t},t})] - F^* & \le \frac{\norm{x_{K_{t-1},t-1} -
\overline{x}_{K_{t-1},t-1}}^2}{2 \gamma_t K_{t}} + \frac{\gamma_t
{\cal B}^2}{2} \\
& \overset{\eqref{th_rspg_linear_rel1}}{\le}
\frac{\epsilon_{t-1}^{2/\nu}}{2 \mu^{2/\nu} \gamma_t K_{t}} + \frac{\gamma_t
{\cal B}^2}{2} \leq \frac{\epsilon_{t-1}}{4} +
\frac{\epsilon_{t-1}}{4} = \epsilon_t,
\end{align*}
which confirms the induction. Since we require $T = \lceil \log(\epsilon_0 /\epsilon) \rceil$ epochs, then we can bound the total number of SPG/SPP  iterations as:
$$\sum_{t=1}^T K_t \leq \frac{4 {\cal B}^2}{\mu^{2/\nu} \epsilon^{2 - 2/\nu}}  \log\left(\frac{\epsilon_0}{\epsilon}\right),$$ which leads to our statement.  \qed
\end{proof}

\noindent Notice that relation \eqref{rspg_sm} for $\nu=1$ becomes the so-called  sharp minima condition \cite{Pol:87}. The previous theorem shows that  we can still achieve linear convergence for SFO schemes, even when the  condition $B=0$ does not hold, provided that the function has a sharp minimum. We only need to implement a restarting variant of these schemes which allows us to obtain linear convergence for the class of  functions with bounded subgradients (i.e. $B \not = 0$) and having a sharp minima type property (i.e. satisfying  the $\nu=1$ functional growth condition \eqref{rspg_sm}). Moreover, this theorem improves the previous  convergence rates of SFO for constant stepsize when $B \not = 0$ and $\nu<2$. More precisely, from Theorem \ref{th:spg_mu0} we have convergence rate ${\cal O}(1/\epsilon^2)$ for SPP/SPG on the class of functions with $B \not = 0$ and $\nu<2$, while the corresponding restarted variants under a $\nu \in (1, \; 2)$ functional growth condition improves to ${\cal O}(1/\epsilon^{2-2/\nu})$. There are several examples of functions satisfying the sufficient conditions of the previous theorem.   It  particular, sharp minima condition has been considered frequently in the literature for proving linear convergence of some first order methods, see e.g. \cite{BurFer:93,Pol:87,YanLin:16}. There are
multiple examples of practical optimization models which satisfy
\eqref{rspg_sm} with $\nu=1$, in particular objective functions whose epigraph is a polyhedral~set:  $F(x) = \max_{\xi \in \Omega}(a_\xi^T x + b_\xi)$.  These
type of problems arise frequently   e.g. in  machine learning
applications and for discrete stochastic models, that is  $f(x) =
1/N\sum_{i=1}^N f(x,i)$ and $g(x) = \|x\|_1$ or the indicator
function of some polyhedral set  $C$. We list below the most
relevant classes of functions from machine learning applications
satisfying \eqref{rspg_sm}: hinge loss $f(x,\xi) = \max(0, 1- y_\xi z^T_\xi x)$,  absolute loss $f(x,\xi) = |z^T_\xi x - y_\xi|$ or $\delta$-insensitive loss  $f(x,\xi) = \max(|z^T_\xi x - y_\xi| - \delta, 0)$. All these  instances satisfy  also \eqref{rspg_b} if  $g(x,\xi)$ is $\norm{x}_1$ or the indicator function of some polyhedral set  $C_\xi$.

\vspace{0.2cm}

\noindent \textbf{Conclusions}: In this paper  we have derived   convergence rates for stochastic first order (SFO)  methods with constant or variable stepsize under general assumptions on the optimization problem  covering a large class of objective functions. In particular,  our analysis covered the   class of  non-smooth Lipschitz  functions and composition of a (potentially) non-smooth function  and a smooth function, with or without strong convexity. Moreover, given some desired accuracy $\epsilon$, in general, our SFO algorithms are $\epsilon$-free and parameters-free, i.e. the stepsize $\gamma_t$ depends on parameters that are easy to compute. E.g., $\gamma_t$ depends on   $L$ but not on $\mu$, see Theorems  2.6 and 2.7  and also Remark 2.2.  


\end{document}